\documentclass[11pt]{amsart}
\usepackage{amsmath,amssymb}
\usepackage{color}
\usepackage[T1]{fontenc}
\usepackage{caption}

\begin{document}
\newcommand{\chp}{\mathds{1}}

\newtheorem{thm}{Theorem}
\newtheorem{theorem}{Theorem}[section]
\newtheorem{lem}[thm]{Lemma}
\newtheorem{lemma}[thm]{Lemma}
\newtheorem{prop}[thm]{Proposition}
\newtheorem{proposition}[thm]{Proposition}
\newtheorem{cor}[thm]{Corollary}
\newtheorem{defn}[thm]{Definition}
\newtheorem*{remark}{Remark}
\newtheorem{conj}[thm]{Conjecture}


\newcommand{\Z}{{\mathbb Z}} 
\newcommand{\Q}{{\mathbb Q}}
\newcommand{\R}{{\mathbb R}}
\newcommand{\C}{{\mathbb C}}
\newcommand{\N}{{\mathbb N}}
\newcommand{\FF}{{\mathbb F}}
\newcommand{\fq}{\mathbb{F}_q}
\newcommand{\rmk}[1]{\footnote{{\bf Comment:} #1}}

\newcommand{\bfA}{{\boldsymbol{A}}}
\newcommand{\bfY}{{\boldsymbol{Y}}}
\newcommand{\bfX}{{\boldsymbol{X}}}
\newcommand{\bfZ}{{\boldsymbol{Z}}}
\newcommand{\bfa}{{\boldsymbol{a}}}
\newcommand{\bfy}{{\boldsymbol{y}}}
\newcommand{\bfx}{{\boldsymbol{x}}}
\newcommand{\bfz}{{\boldsymbol{z}}}
\newcommand{\F}{\mathcal{F}}
\newcommand{\Gal}{\mathrm{Gal}}
\newcommand{\Fr}{\mathrm{Fr}}
\newcommand{\Hom}{\mathrm{Hom}}
\newcommand{\GL}{\mathrm{GL}}

\renewcommand{\mod}{\;\operatorname{mod}}
\newcommand{\ord}{\operatorname{ord}}
\newcommand{\TT}{\mathbb{T}}
\renewcommand{\i}{{\mathrm{i}}}
\renewcommand{\d}{{\mathrm{d}}}
\renewcommand{\^}{\widehat}
\newcommand{\HH}{\mathbb H}
\newcommand{\Vol}{\operatorname{vol}}
\newcommand{\area}{\operatorname{area}}
\newcommand{\tr}{\operatorname{tr}}
\newcommand{\norm}{\mathcal N} 
\newcommand{\intinf}{\int_{-\infty}^\infty}
\newcommand{\ave}[1]{\left\langle#1\right\rangle} 
\newcommand{\Var}{\operatorname{Var}}
\newcommand{\Prob}{\operatorname{Prob}}
\newcommand{\sym}{\operatorname{Sym}}
\newcommand{\disc}{\operatorname{disc}}
\newcommand{\CA}{{\mathcal C}_A}
\newcommand{\cond}{\operatorname{cond}} 
\newcommand{\lcm}{\operatorname{lcm}}
\newcommand{\Kl}{\operatorname{Kl}} 
\newcommand{\leg}[2]{\left( \frac{#1}{#2} \right)}  
\newcommand{\Li}{\operatorname{Li}}

\newcommand{\sumstar}{\sideset \and^{*} \to \sum}

\newcommand{\LL}{\mathcal L} 
\newcommand{\sumf}{\sum^\flat}
\newcommand{\Hgev}{\mathcal H_{2g+2,q}}
\newcommand{\USp}{\operatorname{USp}}
\newcommand{\conv}{*}
\newcommand{\dist} {\operatorname{dist}}
\newcommand{\CF}{c_0} 
\newcommand{\kerp}{\mathcal K}

\newcommand{\Cov}{\operatorname{cov}}
\newcommand{\Sym}{\operatorname{Sym}}

\newcommand{\ES}{\mathcal S} 
\newcommand{\EN}{\mathcal N} 
\newcommand{\EM}{\mathcal M} 
\newcommand{\Sc}{\operatorname{Sc}} 
\newcommand{\Ht}{\operatorname{Ht}}

\newcommand{\E}{\operatorname{E}} 
\newcommand{\sign}{\operatorname{sign}} 

\newcommand{\divid}{d} 

\title{On the sum of digits of $1/M$ in $\fq[x]$}
\author{Ze\'ev Rudnick}
\address{School of Mathematical Sciences, Tel Aviv University, Tel Aviv 69978, Israel}
 \email{rudnick@tauex.tau.ac.il}
 
\date{\today}

\begin{abstract}For certain primes $p$, the average digit in the expansion of $1/p$ was found to have a deviation from random behaviour related to the class number of the imaginary quadratic field $\Q(\sqrt{-p})$ (Girstmair 1994). In this short note, we observe that for the corresponding problem when we replace the integers by  polynomials over a finite field,   there is never any bias. The argument is elementary. 
\end{abstract}
\maketitle

 \subsection{Integers}

Let $b\geq 2$ be an integer. It has long been known (see e.g. \cite[Chapter VI]{Dickson}) that for primes $p$ such that  $\gcd(p,b)=1$, the base $b$ expansion of $1/p$ is purely periodic 
 \[
 \frac 1p = 0.\overline{a_1,a_2,\dots, a_T}
 \]
 where $a_j = a_j(p) \in \{0,1,\dots,b-1\}$ 
 and the period $T=\ord(b,p)$ is the order of $b$ in the multiplicative group of residues modulo $p$, and thus divides $p-1$. If $b \neq 1 \bmod p$  then the period is at least $2$.  
 

 Define the average digit by 
 \[
 A(p) = \frac 1T\sum_{j=1}^T a_j = \frac 1{p-1} \sum_{j=1}^{p-1} a_j  = \lim_{N\to \infty}\frac 1N  \sum_{j=1}^N a_j
 \]
 In many cases, e.g. if  $T$ is even (which happens for instance  if $T=p-1$ is maximal, that is, if $b$ is a primitive root modulo $p$) then it is easy to see that the average is that of a random string, that is
 \begin{equation}\label{Ap primitive}
 A(p) = \frac{b-1}2,\quad \ord(b,p)=p-1.
 \end{equation}
  But in several interesting cases there is a bias. For instance, Girstmair \cite{GirstmairActa, GirstmairMonthly, Girstmair1995} showed that if $\ord(b,p) = (p-1)/2$ and $p=3\bmod 4$ then 
  \[
  A(p) = \frac{b-1}2-\frac{(b-1)}{p-1}h_p
  \]
  where $h_p$ is the class number of the imaginary quadratic field $\Q(\sqrt{-p})$.  In particular, in this case we obtain $A(p)<(b-1)/2$. See also \cite{MurtyThan, Kalyan}.


  \subsection{Polynomials} 
 In this short note, stimulated by a lecture \cite{CS} of Frank Calegari in the Number Theory Web seminar on August 5, 2021, I investigate the corresponding problem when we replace the integers by  polynomials over a finite field, and find that there is never any bias. The argument is elementary.

Let $\fq$ be a finite field of $q$ elements, and $\fq[x]$ the ring of polynomials with coefficients in $\fq$. For nonzero $f\in \fq[x]$ we define the norm
\[
|f| := \#\fq[x]/(f) = q^{\deg f}
\]
and the Euler totient function $\Phi(f) = \#(\fq[x]/(f))^*$, the size of the multiplicative group modulo $f$. For $P$ irreducible, we have $\Phi(P)=|P|-1$.  

 Let $B(x)\in \fq[x]$ be a polynomial of positive degree. Then for any rational function $f\in \fq(x)$ we have the base $B$ expansion  
 \[
 f(x) = \sum_{j=-\infty}^J a_{-j}(x)B(x)^j
 \]
 with $a_j(x)\in \fq[x]$, $\deg a_j<\deg B$ (or possibly $a_j(x)=0$). The $a_j(x)$ are the ``base $B$ digits of $f$'' and we write
 \[
 f(x)  = a_{-J}a_{-(J-1)} \dots a_0\;.\;a_1a_2 a_3\dots 
 \]
 
 Now let $M\in \fq[x]$ be a polynomial of positive degree, coprime to $B$. 
 As in the case of the integers, one sees that the base $B$ expansion of $1/M(x)$ is periodic:
 \[
 \frac 1M = 0.\overline{a_1a_2 \dots a_T}
 \]
 and the minimal period $T$ is the order of $B$ in the multiplicative group $(\fq[x]/(M))^*$, hence divides $\Phi(M) $.  
 If $B\neq 1 \bmod M$ then the period $T$ is at least $2$, see examples in tables \ref{table:B=x} and \ref{table:B=x(x-1)}.  
 \begin{table}[h!]
\centering  
 \caption{The expansion of $1/M$ in base $B(x)=x$ for $q=2$ for all $M$ of degree $\leq 4$, coprime to $B(B-1)$; all are irreducible except $x^4+x^2+1 = (x^2+x+1)^2$. The possible digits are $\{0,1\} = \FF_2$. }
\label{table:B=x}
\begin{tabular}{ |c|c|c| }  
 \hline
 $M$ & $\ord(x,M)$ & $1/M$    
  \\    \hline
   & & 
 \\  
 $x^2+x+1$ & $3$ & $0.\overline{0,1,1} $  
 \\ [0.5ex]
 $  x^3+x+1$ & $7$ & $0.\overline{0,0,1,0,1,1,1}$  
 \\ [0.5ex]
  $ x^3+x^2+1$ & $7$ & $0.\overline{ 0,0,1,1,1,0,1}$ 
 \\  [0.5ex]  
     $x^4+x^3+1$ & $15 $ &  $0.\overline{0,0,0,1,1,1,1,0,1,0,1,1,0,0,1}$   
     \\ [0.5ex] 
  $x^4+x+1$ & $15$ &  $0.\overline{0,0,0,1,0,0,1,1,0,1,0,1,1,1,1 }$ \\ [0.5ex] 
 $x^4+x^3+x^2+x+1$ & $5$ & $0.\overline{0,0,0,1,1 }$
  \\ [0.5ex] 
  $x^4+x^2+1 = (x^2+x+1)^2$ & 6 &  $0.\overline{0,0,0,1,0,1}$
 \\
 \hline
\end{tabular}

\end{table}

\begin{table}
\begin{center}
\caption{The expansion of $1/P$ in base $B(x)=x(x-1)$ for $q=2$ for all   $P$ of degree $\leq 4$, coprime to $B(B-1)$ (necessarily irreducible). The possible digits are $\{0,1,x,x+1\}$.}
\label{table:B=x(x-1)}
\begin{tabular}{ |c|c|c| }  
 \hline
 $P$   & $1/P$    
  \\    \hline
   &   
  \\ [0.5ex]
 $  x^3+x+1$ & $0.\overline{0,x,1,x,1+x,1+x,1}$  
 \\ [0.5ex]
  $ x^3+x^2+1$   & $0.\overline{0,1+x,1,1+x,x,x,1}$ 
 \\  [0.5ex]  
     $x^4+x^3+1$   &  $0.\overline{0,1,1+x,1,1,x,x,0,1+x,0,x,1+x,1+x,x,1}$   
     \\ [0.5ex] 
  $x^4+x+1$  &  $ 0.\overline{0,1,1}$ \\ [0.5ex] 
 $x^4+x^3+x^2+x+1$  & $ 0.\overline{0,1,x,1,1,1+x,1+x,0,x,0,1+x,x,x,1+x,1}$
 \\
 \hline
\end{tabular}
 \end{center}
 
\end{table}

 \subsection{Sum of digits}
 We define the sum of digits function by 
 \[
 S_B(M)(x) = a_1(x)+a_2(x)+\dots +a_T(x) \in \fq[x]
 \]
  which is a polynomial of degree less than $\deg B$.  If $M=P$ is irreducible, then    $T$ divides $\Phi(P)=|P|-1=q^{\deg P}-1$, which is an integer coprime to the characteristic of the finite field, and so by analogy with the integers, we can define the average digit by 
  \[
  A_B(P)(x) =\frac 1T S_B(P)(x) = \frac 1{|P|-1} \sum_{j=1}^{|P|-1} a_j(x).
  \]

If $B=1\bmod M$, that is $M$ divides $B-1$, then writing $B-1=KM$ gives an expansion with period one:
\[
\frac 1M = 0.\overline{K}
\]
and $S_B(M)=K$.   
It turns out that in all other cases, the sum of digits (hence the average digit $A_B(P)$ for $P$ prime) always vanishes: 
\begin{thm}\label{prop:S is zero}
If $M$ is coprime to $B(B-1)$ then $S_B(M)(x)=0$. 
\end{thm}
 \begin{proof}
The proof is elementary and straightforward. We write
\[
B(x)^T-1=M(x)K(x)
\]
and use the base $B$ expansion of $K$:
\[
K(x) = k_0(x) + k_1(x)B+\dots +k_{T-1}(x)B(x)^{T-1}
\]
where $\deg k_j<\deg B$.  
Then the ``digits'' of $1/M$ are simply
\[
a_1=k_{T-1}, \; a_2=k_{T-2}, \dots , a_T = k_0
\]
and 
\begin{equation}\label{lem:S vs K}
S(x) = \sum_{j=0}^{T-1} k_j(x).
\end{equation}

This is because
\[
\begin{split}
\frac 1M &= \frac{K}{B^T-1} = \frac{K}{B^T}\left(1+ \frac 1{B^T} + \frac 1{B^{2T}}+\dots \right)  
\\
&= 
\frac{k_0+k_1B+\dots +k_{T-1}B^{T-1}}{B^T} \left(1+ \frac 1{B^T} + \frac 1{B^{2T}}+\dots \right)   
\\
&= \frac{k_{T-1}}{B}+\frac{k_{T-2}}{B^2} +\dots
=0. \overline{k_{T-1}k_{T-2}\dots k_0} .
\end{split}
\]
Therefore we find
\[
S(x) = \sum_{j=0}^{T-1} k_j(x).
\]

 
 Next, note that $K= S \bmod B-1$, because $B=1 \mod B-1$ and so 
 \[
 K=\sum_j k_jB^j=\sum k_j \bmod B-1
 \]
 which equals $S$ by  \eqref{lem:S vs K}. 
 
 On the other hand, 
 \[
 KM = B^T-1 =(B-1)(B^{T-1}+\dots +1) = 0\bmod B-1
 \]
 and so $B-1 \mid KM$. But by assumption, $\gcd(M,B-1)=1$  and  so $B-1\mid K$, and therefore
 \[
 S=K=0 \bmod B-1.
 \] 
  But $\deg S<\deg B = \deg(B-1)$ while $B-1 \mid S$ which implies $S(x)=0$.
 \end{proof}



\begin{thebibliography}{1}

\bibitem{Dickson}
 L. E. Dickson, History of the Theory of Numbers, Vol. I (reprint), Chelsea, New York, 1952.

\bibitem{CS}
Frank Calegari, ``Digits''. Number Theory Web seminar, August 5, 2021, 
https://sites.google.com/view/ntwebseminar/previous-talks?authuser=0

\bibitem{GirstmairActa}
Girstmair, Kurt  
The digits of $1/p$ in connection with class number factors.
Acta Arith. 67 (1994), no. 4, 381--386.

\bibitem{GirstmairMonthly}
Girstmair, Kurt  
A ``popular'' class number formula.
Amer. Math. Monthly 101 (1994), no. 10, 997--1001.

\bibitem{Girstmair1995}
Girstmair, Kurt
Periodische Dezimalbr\"{u}che---was nicht jeder dar\"{u}ber wei\ss.  Jahrbuch \"{U}berblicke Mathematik, 1995, 163--179, Friedr. Vieweg, Braunschweig, 1995.


\bibitem{Kalyan}
Kalyan Chakraborty, Krishnarjun Krishnamoorthy. 
On some symmetries of the base n expansion of $1/m$: The Class Number connection.
	arXiv:2107.12123 [math.NT]
	
\bibitem{MurtyThan}
Murty, M. Ram ; Thangadurai, R.  
The class number of $\Q(\sqrt{-p})$ and digits of $1/p$. 
Proc. Amer. Math. Soc. 139 (2011), no. 4, 1277--1289.

\end{thebibliography}
\end{document}